\documentclass{amsart}

\theoremstyle{definition}

\theoremstyle{remark}

\numberwithin{equation}{section}

\theoremstyle{definition}\newtheorem{thm}{Theorem}[section]
\theoremstyle{definition}\newtheorem{cor}[thm]{Corollary}
\theoremstyle{definition}\newtheorem{lem}[thm]{Lemma}
\theoremstyle{definition}\newtheorem{prop}[thm]{Proposition}
\theoremstyle{definition}
\theoremstyle{definition}
\theoremstyle{remark}
\theoremstyle{definition}
\theoremstyle{definition}

\begin{document}
%\begin{large}
\title{functionally countable subalgebras and some properties of Banaschewski compactification}

\author{AliReza Olfati}
\address{Department of Mathematics, Yasouj University, Yasouj, Iran}
%\curraddr{Department of Mathematics and Statistics, Case Western
%Reserve University, Cleveland, Ohio 43403}
\email{alireza.olfati@yu.ac.ir, olfati.alireza@gmail.com}

%\author{Alireza Olfati}
%\address{Department of Mathematics, Yasouj University, Yasouj,
%Iran} \email{ataherifar@mail.yu.ac.ir, ataherifar54@gmail.com}
%    General info

%\date{January 1, 2001 and, in revised form, June 22, 2001.}



\begin{abstract}
Let $X$ be a zero-dimensional space and $C_c(X)$ be the set of all continuous real valued functions on $X$ with countable image. In this article we denote by $C_c^K(X)$ (resp., $C_c^{\psi}(X))$ the set of all functions in $C_c(X)$ with compact  (resp., pseudocompact) support. First, we observe that $C_c^K(X)=O_c^{\beta_0X\setminus X}$ (resp., $C^{\psi}_c(X)=M_c^{\beta_0X\setminus \upsilon_0X}$). This implies that for an  $\Bbb{N}$-compact space $X$,  the intersection of all free maximal ideals in $C_c(X)$ equals to $C_c^K(X)$, i.e., $M_c^{\beta_0X\setminus X}=C_c^K(X)$. Afterwards, by applying methods of functionally countable subalgebras, we observe some results in remainder of Banaschewski compactification. It is shown that for a zero-dimensional non pseudocompact space $X$, the set $\beta_0X\setminus \upsilon_0X$ has cardinality at least $2^{2^{\aleph_0}}$. Moreover, for a locally compact and $\Bbb{N}$-compact space $X$, the remainder $\beta_0X\setminus X$ is an almost $P$-space. These results leads us to find a class of Parovi$\breve{\mbox{c}}$enko spaces in Banaschewski compactification of a non pseudocompact zero-dimensional space. We conclude with  a theorem which gives a lower bound for the cellularity of subspaces $\beta_0X\setminus \upsilon_0X$ and $\beta_0X\setminus X$, whenever $X$ is a zero-dimensional, locally compact space which is not pseudocompact.

\end{abstract}

\maketitle
\section{Preliminaries}

We recall that a zero-dimensional topological space is a Hausdorff space with a base consisting of clopen sets. Mrowka showed that $X$ is zero-dimensional if and only if it can be embedded into the product space $\Bbb{N}^{\kappa}$, where $\Bbb{N}$ is the set of natural numbers with discrete topology and ${\kappa}$ an arbitrary cardinal number.

We also recall that a topological space $X$ is $\Bbb{N}$-compact if it can be embedded as a closed subset of the product space $\Bbb{N}^{\kappa}$, for some cardinal number ${\kappa}$,  see \cite{EM, M1, M2, M3, M4, M5} for more details on this subject.

 For every zero-dimensional space $X$, there exists an $\Bbb{N}$-compact space $\upsilon_0X$ such that $X$ is dense in it and every continuous function $f:X\rightarrow Y$, with $Y$ an $\Bbb{N}$-compact space, has a unique extension $f^*:\upsilon_0X\rightarrow Y$. We can replace an arbitrary $\Bbb{N}$-compact space $Y$ by the fixed discrete space $\Bbb{Z}$ (the set of integer numbers) and have the following characterization of $\Bbb{N}$-compactification of a zero-dimensional space, see \cite [5.4 (d)]{PW}.

\begin{thm}\label{1.1}
An $\Bbb{N}$-compact extension $T$ of a zero-dimensional space $X$ is homeomorphic to $\upsilon_0X$ if and only if for each continuous function $f:X\rightarrow \Bbb{Z}$, there exists $F:T\rightarrow \Bbb{Z}$ such that $F|_{X}=f$.
\end{thm}

We remind that a Tychonoff space $X$ is strongly zero-dimensional if and only if every two disjoint zero-sets in $X$ are separated by a clopen partition. It is well known that every strongly zero-dimensional real compact space is $\Bbb{N}$-compact. Since every countable subset of $\Bbb{R}$ (the set of real numbers) is Lindel\"{o}ff and zero-dimensional, it should be strongly zero-dimensional. This fact implies that every countable subset of $\Bbb{R}$ is $\Bbb{N}$-compact. So we have the following lemma.

\begin{lem}\label{1.2}
Let $X$ be a zero-dimensional Hausdorff space. For each continuous function $f:X\rightarrow \Bbb{R}$ with countable image, there exists an extension $f^*:\upsilon_0X\rightarrow \Bbb{R}$ such that the image of $f^*$ equals to the image of $f$.
\end{lem}

For an arbitrary Tychonoff space $X$, we denote by $C_c(X)$ the set of all continuous real valued functions on $X$ with countable image. The set $C_c(X)$ forms a subring with pointwise addition and multiplication. Ghadermazi, Karamzadeh and Namdari showed in \cite{GKN} that for a Tychonoff space $X$ there exists a zero-dimensional space $Y$ such that $C_c(X)\cong C_c(Y)$ as a ring isomorphism. Motivated by this fact, in the present article we restrict our attension to zero-dimensional spaces. The reader could find all prerequisites and unfamiliar notions for this subring in \cite{GKN}.

For a zero-dimensional space $X$, let $\beta_0X$ denote its Banaschewski compactification. We recall that $\beta_0X$ is the unique (up to homeomorphism) zero-dimensional compact space containing $X$ as a dense subset such that every continuous two-valued function $f:X\rightarrow \{0,1\}$ has an extension to $\beta_0X$, see \cite[4.7, Corollary (f)]{PW}. It was shown that $X$ is strongly zero-dimensional if and only if $\beta_0X$ is homeomorphic to $\beta X$, i.e., $\beta X$ is zero-dimensional, see \cite{E}. 

Dowker gave an example of a zero-dimensional space $X$ in which $\beta X$ is not zero-dimensional and hence $\beta X\neq \beta_0X$, see \cite[exercise 4V]{PW}. The structure of $\beta_0X$ is related to the clopen ultrafilters defined on $X$. Indeed $\beta_0X$ is homeomorphic to the set of all clopen ultrafilters equipped with the stone topology, see \cite[4.7]{PW}. In \cite[exercise 5E]{PW}, it was also given an outline for recovering $\upsilon_0X$ as a subspace of all clopen ultrafilters on $X$ which have countable intersection property. Therefore we have $X\subseteq \upsilon_0X\subseteq \beta_0X$. Note that for $p\in \beta_0X\setminus \upsilon_0X$, there exists a sequence $\{V_n:n\in \Bbb{N}\}$ of clopen neighborhoods of $p$ in $\beta_0X$ such that $\cap_{n=1}^{\infty}V_n$ does not meet $X$.\\
This article consists of two parts. In part one, by applying previous notions, we characterize some important ideals in $C_c(X)$. It is shown that the set of all functions in $C_c(X)$ with compact (resp., pseudocompact)  support equals to the ideal $O_c^{\beta_0X\setminus X}$ (resp., $M_c^{\beta_0X\setminus \upsilon_0X}$). It is shown that for an $\Bbb{N}$-compact space $X$, the intersection of all free maximal ideals in $C_c(X)$ equals to the ideal $C_c^K(X)$. In part two, we apply the subring $C_c(X)$ and some related notions to find some information about $\beta_0X$. For example, we will give the least cardinality of the remainder $\beta_0X\setminus \upsilon_0X$, i.e., the set of all clopen ultrafilters on $X$ which do not have countable intersection property. It is shown that for a non pseudocompact space $X$, there are at least $2^{2^{\aleph_0}}$ such clopen ultrafilters on $X$. We shall show that if $X$ is locally compact and $\Bbb{N}$-compact, then $\beta_0X\setminus X$ is an almost $P$-space, i.e., a space for which the interior of every zero-set is nonempty, see \cite{L} and \cite{V}. We show that zero-sets of $\beta_0X$ which do not meet $X$ are Parovi$\breve{\mbox{c}}$enko spaces. In the end, it is shown that whenever $X$ is zero-dimensional, locally compact which is not pseudocompact, the cellularity of the spaces $\beta_0X\setminus \upsilon_0X$ and $\beta_0X\setminus X$  are at least $2^{\aleph_0}$. We close this section with the following two results which are useful in the rest of this article.

\begin{lem}\label{1.3}
(a) For a sequence $\{U_n:n\in \Bbb{N}\}$ of clopen subsets of $X$, there exists $f\in C_c(X)$ such that $Z(f)=\cap_{n=1}^{\infty}U_n$.\\
(b) If $f\in C_c(X)$, then there exists a sequence $\{W_n:n\in \Bbb{N}\}$ of clopen subsets of $X$ such that $Z(f)=\cap_{n=1}^{\infty}W_n$.
\end{lem}

\begin{proof}
(a) With out lose of generality, we could assume that $U_1\supseteq U_2\supseteq U_3\cdots$ be a decreasing sequence of clopen sets of $X$. Now define $$f=\sum_{n=1}^{\infty}\frac{1}{2^n}\chi_{_{\left(X\setminus U_n\right)}},$$
where $\chi_{_{X\setminus U_n}}$ is the characteristic function of the clopen set $X\setminus U_n$.
 It is easy to see that $f(X)\subseteq\{0\}\cup\{\frac{1}{2^n}:n\in \Bbb{N}\cup\{0\}\}$. Therefore $f\in C_c(X)$ and $Z(f)=\cap_{n=1}^{\infty}U_n$.\\
(b) Suppose that $Z\in Z_c[X]$. Consider $0<f\in C_c(X)$ such that $Z=Z(f)$. Choose a decreasing sequence $r_1>r_2>\cdots>r_n>\cdots$ of real numbers which tends to zero and for each $n\in \Bbb{N}$, $r_n\notin f(X)$. For each $n\in \Bbb{N}$, define $W_n=f^{-1}[0,r_n)$. Then for each $n\in \Bbb{N}$, $W_n$ is clopen in $X$ and $Z(f)=\cap_{n=1}^{\infty}W_n$. So we are done.
\end{proof}

\begin{cor}\label{1.4}
Let $X$ be a zero-dimensional space. For $f\in C_c(X)$ there exists $F\in C_c(\beta_0X)$ such that $Z(f)=Z(F)\cap X$.
\end{cor}

\begin{proof}
By part (b) of Lemma \ref{1.3}, there exists a sequence $\{W_n:n\in \Bbb{N}\}$ of clopen subsets of $X$ such that $Z(f)=\cap_{n=1}^{\infty}W_n$. For each $n\in \Bbb{N}$, $cl_{\beta_0X}W_n$ is a clopen subset in $\beta_0X$. Part (a) of Lemma \ref{1.3} implies that there exists $F\in C_c(X)$ such that $Z(F)=\bigcap_{n=1}^{\infty}cl_{\beta_0X}W_n$. Clearly $Z(f)=Z(F)\cap X$.
\end{proof}

\section{Characterisation of some special ideals in $C_c(X)$.}

In the begining of this section we characterize maximal ideals of $C_c(X)$. This characterization leads us to specify the ideal which consists of  all functions in $C_c(X)$ with compact (resp., pseudocompact) support. We remind that in \cite{BKM} it is shown that the space of maximal ideals of $C_c(X)$ is isomorphic to $\beta_0X$. But we need to settle the internal characterization of maximal ideals of $C_c(X)$.  In the following  we could observe the counterpart of Gelfand-Kolmogoroff theorem in rings of continuous functions. First, we need a lemma which is essential for characterizing all maximal ideals of this subring.

\begin{lem}\label{2.1}
Let $X$ be zero-dimensional. For each two $f,g\in C_c(X)$ we have $$cl_{\beta_0 X}\left(Z(f)\cap Z(g)\right)=cl_{\beta_0 X}Z(f)\cap cl_{\beta_0 X}Z(g).$$
\end{lem}
\begin{proof}
Case 1.  Suppose $Z(f)\cap Z(g)=\emptyset$. There exists $h\in C_c(X)$ such that $Z(f)=Z(h)$ and $Z(g)=h^{-1}(1)$. Choose $0<r<1$ such that $r\notin h(X)$ and put $U=h^{-1}(-\infty, r)$. The subset $U$ is clopen in $X$, $Z(f)\subseteq  U$ and $Z(g)\subseteq X\setminus U$. Note that $X$ is two-embedded in $\beta_0 X$ and this implies that $cl_{\beta_0 X}U\cap cl_{\beta_0 X}(X\setminus U)=\emptyset$ and then $cl_{\beta_0 X}Z(f)\cap cl_{\beta_0 X}Z(g)=\emptyset$.\\Case 2. Evidently $cl_{\beta_0 X}\left(Z(f)\cap Z(g)\right)\subseteq cl_{\beta_0 X}Z(f)\cap cl_{\beta_0 X}Z(g)$. For $p\in cl_{\beta_0 X}Z(f)\cap cl_{\beta_0 X}Z(g)$, assume that $p$ does not belong to $cl_{\beta_0 X}\left(Z(f)\cap Z(g)\right)$. There exists a clopen set $U\subseteq \beta_0 X$ such that $p\in U$ and $U\cap\left(cl_{\beta_0 X}\left(Z(f)\cap Z(g)\right)\right)=\emptyset$. The subset $V=U\cap X$ is clopen in $X$ and also is the zero-set of the characteristic function $\chi_{_{X\setminus V}}$, i.e., $Z(\chi_{_{X\setminus V}})=V$. Then $V\cap Z(f)\cap Z(g)=\emptyset$. By case 1, we have $$cl_{\beta_0 X}(V\cap Z(f))\cap cl_{\beta_0 X}(V\cap Z(g))=\emptyset,$$
and hence $p$ does not belong to at least one of them, say for example $p\notin cl_{\beta_0 X}(V\cap Z(f))$. Choose a neighborhood $W$ of $p$ such that $W\cap V\cap Z(f)=\emptyset$ and hence $W\cap U\cap Z(f)=\emptyset$. But $W\cap U$ is a neighborhood of $p$ and must intersects $Z(f)$. This is a contradiction and the proof is complete.
\end{proof}

\begin{thm}\label{2.2}
[Gelfand-Kolmogoroff] The Maximal ideals of $C_c(X)$ are in one to one correspondence with the points of $\beta_0 X$ and are given by $$M^p_c=\{f\in C_c(X): p\in cl_{\beta_0 X}Z(f)\},$$
for $p\in \beta_0 X$.
\end{thm}

\begin{proof}
First we show that for each $p\in \beta_0 X$, $M^p_c$ is a maximal ideal. By Lemma \ref{2.1}, one can verify that $M^p_c$ forms an ideal. Suppose on the contrary that there exists a maximal ideal $M$ which properly contains $M^p_c$. Therefore there exists $f\in M$ such that $p\notin cl_{\beta_0 X}Z(f)$. Note that $\beta_0 X$ is zero-dimensional. Hence there exists a clopen subset $U$ in $\beta_0 X$ such that $p\in U$ and $U\cap cl_{\beta_0 X}Z(f)=\emptyset$. The set $V=U\cap X$ is clopen in $X$ and clearly $p\in cl_{\beta_0 X}V$. Since $V$ is the zero-set of the characteristic function $\chi_{_{X\setminus V}}\in C_c(X)$, therefore $\chi_{_{X\setminus V}}\in M^p_c$. Consider the function $g=\chi_{_{X\setminus V}}+f^2$. Obviously $Z(g)=Z(\chi_{_{X\setminus V}})\cap Z(f)=V\cap Z(f)=\emptyset$. Therefore $M$ contains a unit of $C_c(X)$ which implies that $M=C_c(X)$.\\ Now we must show that each maximal ideal $M$ in $C_c(X)$ has this form. By Lemma \ref{2.1}, the set $\{cl_{\beta_0 X}Z(f): f\in M\}$ is a family of closed subsets with the finite intersection property and since $\beta_0 X$ is compact, there exists $p\in \bigcap_{f\in M}cl_{\beta_0 X}Z(f)$. Therefore $M\subseteq M^p_c$ and hence $M=M^p_c$.
\end{proof}

Suppose that $p\in \beta_0X$. We define the set $O^p_c$ as follows. 
$$O^p_c=\{f\in C_c(X): p\in int_{\beta_0X}cl_{\beta_0X}Z(f)\}.$$ One could observe that $O^p_c$ forms an ideal of $C_c(X)$ and the only maximal ideal containing $O^p_c$ is $M^p_c$. For $p\in \beta_0X$, in the following we characterize the set $Z[O^p_c]$. Note that in general, all bounded functions on $X$ could not be extended to $\beta_0X$. But the following proposition allows us to extend zero-sets of $O^p_c$.

\begin{prop}\label{2.3}
For each $p\in \beta_0X$, $$Z[O^p_c]=\{Z'\cap X:  Z'\in Z_c[\beta_0X], p\in int_{\beta_0X}Z'\}.$$
\end{prop}

\begin{proof}
Suppose that $Z\in Z[O^p_c]$, i.e., $p\in int_{\beta_0X}cl_{\beta_0X}Z$. By Corollary \ref{1.4}, there exists $Z'\in Z_c[\beta_0X]$ such that $Z'\cap X=Z$. Clearly $p\in int_{\beta_0X}Z'$. Hence $Z$ belongs to the right hand side set. Now let $Z'\in Z_c[\beta_0X]$ and $p\in int_{\beta_0X}Z'$. There exists a clopen set $U\subseteq \beta_0X$ such that $p\in U\subseteq int_{\beta_0X}Z'$. Therefore $U\cap X\subseteq Z'\cap X$ and hence $p\in int_{\beta_0X}cl_{\beta_0X}(Z'\cap X)$. This implies that $Z'\cap X\in Z[O^p_c]$.
\end{proof}

We denote by $C_c^K(X)$, the set of all functions of $C_c(X)$ with compact support. In the following we characterize the set $C_c^K(X)$ as an ideal of $C_c(X)$. 

\begin{thm}\label{2.4}
For a zero-dimensional space $X$, $C_c^K(X)=O_c^{\beta_0X\setminus X}$.
\end{thm}

\begin{proof}
Suppose that $f\in C_c^K(X)$. Then $cl_{X}\left(X\setminus Z(f)\right)$ is compact and hence $cl_{\beta_0X}\left(X\setminus Z(f)\right)\subseteq X$. By Corollary \ref{1.4}, there exists $Z'\in Z_c[\beta_0X]$ such that $Z(f)=Z'\cap X$. Therefore $X\cap(\beta_0X\setminus Z')=X\setminus Z(f)$. Since $X$ is dense and $\beta_0X\setminus Z'$ is open in $\beta_0X$, $cl_{\beta_0X}\left(\beta_0X\setminus Z'\right)\subseteq X$. Thus $\beta_0X\setminus X\subseteq \beta_0X\setminus cl_{\beta_0X}\left(\beta_0X\setminus Z'\right)=int_{\beta_0X}Z'$. Hence by Proposition \ref{2.3}, for all $p\in \beta_0X\setminus X$ , $Z'\cap X\in Z_c[O^p_c]$ and therefore $f\in O^{\beta_0X\setminus X}_c$. For the reverse inclusion, suppose that $f\in O^{\beta_0X\setminus X}_c$. Then $\beta_0X\setminus X\subseteq int_{\beta_oX}cl_{\beta_0X} Z(f)$. By Proposition \ref{2.3}, there exists $Z'\in Z_c[\beta_0X]$ such that $Z'\cap X=Z(f)$ and $\beta_0X\setminus X\subseteq int_{\beta_0X} Z'$. This implies that $\beta_0X\setminus int_{\beta_0X}Z'\subseteq X$. Since $\beta_0X\setminus int_{\beta_0X}Z'= cl_{\beta_0X}(\beta_0X\setminus Z')\subseteq X$ and $X\setminus Z(f)\subseteq \beta_0X\setminus Z'$, clearly $cl_{X}(X\setminus Z(f))$ is compact and hence $f\in C_c^K(X)$.
\end{proof}

The following two results are important for the rest of this section. We recall that a subset $S\subseteq X$ is $C_c$-embedded in $X$ if for each $f\in C_c(S)$, there exists $F\in C_c(X)$ such that $F|_X=f$.

\begin{prop}\label{2.5}
If $S\subseteq X$ is $C_c$-embedded in $X$, then it is separated by a clopen partition from every $Z\in Z_c[X]$ disjoint from it.
\end{prop}

\begin{proof}
Suppose that $h\in C_c(X)$ and $S\cap Z(h)=\emptyset$. Define $f(s)=\frac{1}{h(s)}$ for all $s\in S$. Let $F\in C_c(X)$ be such that $F|_{S}=f$. Put $k=hF$. Clearly $k\in C_c(X)$ and $k|_{S}=1$ and $k|_{Z(h)}=0$. It is enough to choose $0<r<1$ such that $r\notin k(X)$. Then $U=k^{-1}(-\infty, r)$ is a clopen subset of $X$ that $Z(h)\subseteq U$ and $S\subseteq X\setminus U$.
\end{proof}

In the following we need to know about the existence of $C_c$-embedded subsets. Note that $C$-embedded subsets are not applicable for our purpose. Because we do not know anything  about the cardinality of the image of the extension of a function with countable image. The following lemma gives us a condition for existing a $C_c$-subset in a zero-dimensional space. 

\begin{lem}\label{2.6}
Let $X$ be zero-dimensional and $f\in C_c(X)$ carries a subset $S\subseteq X$ homeomorphically to a closed subset $f(S)\subseteq f(X)$. Then $S$ is $C_c$-embedded in $X$.
\end{lem}
\begin{proof}
Suppose that $f|_S:S\rightarrow f(S)$ is a homeomorphism. Therefore $f^{-1}:f(S)\rightarrow S$ is continuous. Consider $g\in C_c(S)$. We observe that the composition  $gof^{-1}$ belongs to $C_c(f(S))$.  The subset $f(X)$ is countable and therefore normal. Hence $gof^{-1}$ has an extension $G$ to all of $f(X)$. It is obvious to see that $G\in C_c(f(X))$. The composition $Gof$ belongs to $C_c(X)$. For $s\in S$, we have $Gof(s)=G(f(s))=gof^{-1}(f(s))=g(s)$. Therefore $Gof$ is the extension of $g$ to $X$ which has countable image.
\end{proof}

The previous lemma leads us to the following corollary.

\begin{cor}\label{2.7}
For a zero-dimensional space $X$, let $E\subseteq X$. Suppose the function $h\in C_c(X)$ is unbounded on $E$. Then $E$ contains a copy of $\Bbb{N}$, which is $C_c$-embedded in $X$ and $h$ approach to infinity on $E$.
\end{cor}

In the following, we adapt the original approach of Mandelker in \cite{Man1} and \cite{Man2} to functionally countabe subalgebras for  characterizing the subset consisting of all functions in $C_c(X)$ with pseudocompact support. We denote by $C_c^{\psi}(X)$, the set of all functions in $C_c(X)$ with pseudocompact support. We recall that a subset $S\subseteq X$ is relatively pseudocompact with respect to $C_c(X)$, if for each $f\in C_c(X)$, the function $f|_{S}$ is bounded.  We have the following equivalence for relatively pseudocompact subsets with respect to $C_c(X)$.

\begin{prop}\label{2.7.5}
Let $X$ be a zero-dimensional space. The set $A\subseteq X$ is relatively pseudocompact  with respect to $C_c(X)$ if and only if $cl_{\beta_0X}A\subseteq \upsilon_0X$.
\end{prop}

\begin{proof}
For the necessity, suppose $cl_{\beta_0X}A\cap(\beta_0X\setminus \upsilon_0X)\neq\emptyset$ and choose $p\in cl_{\beta_0X}A\cap(\beta_0X\setminus \upsilon_0X)$. There exists a sequence $\{V_n:n\in\Bbb{N}\}$ of clopen neighborhoods of $p$ in $\beta_0X$ such that $(\cap_{n=1}^{\infty}V_n)\cap \upsilon_0X=\emptyset$. By part (a) of Lemma \ref{1.3}, there exists $F\in C_c(\beta_0X)$ such that $Z(F)=\cap_{n=1}^{\infty}V_n$. Put $f=F|_X$ and define $h=\frac{1}{f}$. Clearly $h\in C_c(X)$ and since $p$ is a limit point of $A$ in $\beta_0X$, $h$ is unbounded on $A$, which is a contradiction. For the sufficiency, consider $f\in C_c(X)$. The extension $f^{\upsilon_0}\in C_c(\upsilon_0X)$ is bounded on the compact subset $cl_{\beta_0X}A$. Therefore $f$ should be bounded on $A$.
\end{proof}

For continuing our progress, we need a proposition which is due to Pierce. The reader could find it in \cite[Lemma 1.9.3]{PI}.
\begin{prop}\label{2.8}
A zero-dimensional space $X$ is not pseudocompact if and only if there exists a continuous and onto map $f:X\rightarrow \Bbb{N}$.
\end{prop}

\begin{thm}\label{2.9}
If $f\in C_c(X)$ and $X\setminus Z(f)$ is relatively pseudocompact with respect to $C_c(X)$, then $X\setminus Z(f)$ is pseudocompact.
\end{thm}

\begin{proof}
Suppose on the contrary that $S=cl_{X}(X\setminus Z(f))$ is not pseudocompact. Hence by Proposition \ref{2.8}, there exists $h\geq 1$ and $h\in C_c(S)$ that is unbounded on $X\setminus Z(f)$. By Corollary \ref{2.7}, there exists a countable discrete subset  $D\subseteq X\setminus Z(f)$ such that $D$ is $C_c$-embedded in $S$ and $h$ is unbounded on $D$. Proposition \ref{2.5} implies that there exists a clopen set $O\subseteq S$ such that $D\subseteq O$ and $Z(f)\cap S\subseteq S\setminus O$. Define the function $k$ as follows,
$$ k(x) = \left\{
  \begin{array}{l l}
    \frac{1}{h(s)}\vee \chi_{_{S\setminus O}}(s),& s\in S \\
    1, &  s\in cl_{X}\left(X\setminus S\right).
  \end{array} \right. $$
Pasting lemma implies that $k$ is continuous on $X$. Clearly $k\in C_c(X)$ and $k>0$. Hence $\frac{1}{k}$ is unbounded on $D$. Thus $S$ is not relatively pseudocompact with respect to $C_c(X)$, a contradiction.
\end{proof}

\begin{thm}\label{2.10}
Let $X$ be a zero-dimensional space and $f\in C_c(X)$. If $\beta_0X\setminus \upsilon_0X\subseteq cl_{\beta_0X}Z(f)$, then $\beta_0X\setminus \upsilon_0X\subseteq int_{\beta_0X} cl_{\beta_0X}Z(f)$
\end{thm}

\begin{proof}
Suppose that $p\in\beta_0X\setminus \upsilon_0X$. There exists a sequence $\{V_n: n\in \Bbb{N}\}$ of clopen neighborhoods of $p$ in $\beta_0X$ such that $\left(\cap_{n=1}^{\infty}V_n\right)\cap \upsilon_0X=\emptyset$. By part (a) of Lemma \ref{1.3}, there exists $h\in C_c(\beta_0X)$ such that $ Z(h)=\cap_{n=1}^{\infty}V_n$. Put $T=\beta_0X\setminus Z(h)$. Consider the function $\frac{1}{h}$ on $T$. If $Z(h)$ meets $cl_{\beta_0X}(X\setminus Z(f))$, then $\frac{1}{h}$ which is continuous on $T$ is unbounded on $X\setminus Z(f)$. Hence by Corollary \ref{2.7}, there exists a countable closed set $S\subseteq X\setminus Z(f)$ which is $C_c$-embedded in $T$. Thus $S$ is $C_c$-embedded in $X$ and by Proposition \ref{2.5}, there exists a clopen set $O\subseteq X$ such that $S\subseteq O$ and $Z(f)\subseteq X\setminus O$. Therefore $cl_{\beta_0X}S\cap cl_{\beta_0X}Z(f)=\emptyset$. Since $S$ is closed in $T$ and also is noncompact, there exists $p\in cl_{\beta_0X}S\setminus T$. This implies that $p\in Z(h)$ and $p\notin cl_{\beta_0X}Z(f)$ which is a contradiction. Hence $Z(h)\cap cl_{\beta_0X}(X\setminus Z(f))=\emptyset$ and so $Z(h)\subseteq \beta_0X\setminus cl_{\beta_0X}(X\setminus Z(f))\subseteq cl_{\beta_0X}Z(f)$. Therefore $cl_{\beta_0X}Z(f)$ is a neighborhood of $p$.
\end{proof}

Now we are ready to characterize another important ideal in $C_c(X)$. We denote by  $C_c^{\psi}(X)$, the set of all functions in $C_c(X)$ with pseudocompact support.

\begin{thm}\label{2.11}
Let $X$ be a zero-dimensional space. Then $C_c^{\psi}(X)=M_c^{\beta_0X\setminus \upsilon_0X}=O_c^{\beta_0X\setminus \upsilon_0X}$.
\end{thm}

\begin{proof}
By Theorem \ref{2.10}, the second equality is clear. Assume that $f\in C_c^{\psi}(X)$. Since $S(f)=X\setminus Z(f)$ is pseudocompact and $\left(\upsilon_0X\setminus Z(f^{\upsilon_0})\right)\cap X=X\setminus Z(f)$, the subset $\upsilon_0X\setminus Z(f^{\upsilon_0})$ is also pseudocompact. Note that $S(f^{\upsilon_0})=\upsilon_0X\setminus Z(f^{\upsilon_0})$ is a cozero-set of the real compact space $\upsilon_0X$, and hence it should be realcompact, see \cite[Corollary 8.14]{GJ}. Therefore $S(f^{\upsilon_0})$ must be compact. We observe that $$\beta_0X=cl_{\beta_0X}Z(f)\cup cl_{\beta_0X}S(f)=cl_{\beta_0X}Z(f)\cup S(f^{\upsilon_0}),$$ and hence we have $\beta_0X\setminus \upsilon_0X\subseteq cl_{\beta_0X}Z(f)$. Thus $f\in M_c^{\beta_0X\setminus \upsilon_0X}$. Now if $f\in O_c^{\beta_0X\setminus \upsilon_0X}$, there exists a compact set $K$ such that $$\beta_0X\setminus \upsilon_0X\subseteq \beta_0X\setminus K\subseteq cl_{\beta_0X}Z(f),$$ and hence $X\setminus Z(f)\subseteq \beta_0X\setminus cl_{\beta_0X}Z(f)\subseteq K\subseteq \upsilon_0X$. Since $X\setminus Z(f)$ is relatively pseudocompact with respect to $C_c(X)$, by Theorem \ref{2.9},  the set $X\setminus Z(f)$ is pseudocompact.
\end{proof}

For an $\Bbb{N}$-compact space, Theorem \ref{2.4} and Theorem \ref{2.11} imply the following corollary which says that in an $\Bbb{N}$-compact space, $C_c^K(X)$ equals to that intersection of all free maximal ideals in $C_c(X)$.

\begin{cor}\label{2.12}
If $X$ is an $\Bbb{N}$-compact space, then $C_c^K(X)=M_c^{\beta_0X\setminus X}$.
\end{cor}

%------------------------------------------------------------------------------------
\section{ Remainder of Banaschewski compactification via $C_c(X)$.}

In this section, first we want to find the least cardinality of the remainder $\beta_0 X\setminus \upsilon_0 X$. In other point of view, we show that for a zero-dimensional non pseudocompact space $X$, we have at least $2^{2^{\aleph_0}}$ clopen ultrafilters on $X$ which do not have countable intersection property. In the rest of this section we observe that the remainder of the Banaschewski compactification of an $\Bbb{N}$-compact space is an almost $P$-space and a connection between the remainder and  Parovi$\breve{\mbox{c}}$enko spaces are given. Finally, for a zero-dimensional, locally compact space $X$ which is not pseudocompact, we give a lower bound for the cellularity of spaces $\beta_0X\setminus \upsilon_0X$ and $\beta_0X\setminus X$. The following two lemmas are needed in the sequel. The second one is a consequence of Proposition \ref{2.8}.

\begin{lem}\label{3.1}
Each $C_c$-embedded subset $S$ of $X$ is two-embedded.
\end{lem}
\begin{proof}
Let $f:S\rightarrow \{0,1\}$ be a continuous two-valued function, then $f\in C_c(X)$. Therefore there exists $G\in C_c(X)$ such that $G|_S=f$. The image of $G$ is countable and there exists a real number $0<r<1$ such that $r\notin G(X)$. The subset $U=G^{-1}(\infty, r)$ is clopen in $X$ and $f^{-1}(0)\subseteq U$ and $f^{-1}(1)\subseteq X\setminus U$. Define $F:X\rightarrow \{0,1\}$ to be 0 on $U$ and 1 on $X\setminus U$. The function $F$ is continuous and two-valued whose restriction to $S$ is $f$.
\end{proof}

\begin{lem}\label{3.2}
Let $X$ be a zero-dimensional space. Then $X$ is pseudocompact if and only if $\beta_0 X=\upsilon_0 X$.
\end{lem}

\begin{proof}
For the necessity, suppose that $X$ is pseudocompact and $\beta_0 X\setminus \upsilon_0 X\neq \emptyset$. For $p\in \beta_0 X\setminus \upsilon_0 X$, there exists a countable set consisting of clopen  neighborhoods of $p$, say $\{U_n: n\in \Bbb{N}\}$, in $\beta_0 X$ with $\left(\cap_{n=1}^{\infty}U_n\right) \cap \upsilon_0 X=\emptyset$. By part (a) of Lemma \ref{1.3}, there exists a function $f\in C_c(\beta_0 X)$ such that $Z(f)=\cap_{n=1}^{\infty}U_n$. If we restrict $f$ to $X$, then for the function $g=f|_X\in C_c(X)$, we have $Z(g)=\emptyset$. Hence $g$ is a unit of $C_c(X)$. The function $\frac{1}{g}$ belongs to $C_c(X)$ and clearly $g$ is unbounded on $X$ which is a contradiction.\\ For the sufficiency, suppose that $X$ is not pseudocompact. Hence by Proposition \ref{2.8}, there exists a continuous and onto map $f:X\rightarrow \Bbb{N}$. The function $f$ has countable image and therefore has an extension to $\upsilon_0 X$. But   $\upsilon_0 X$ is compact and the extension of $f$ is unbounded which is a contradiction. This completes the proof of our lemma.
\end{proof}

Now we are ready to present one of our main theorems in this section.  This result is notable whenever we deal with zero-dimensional spaces which are not strongly zero-dimensional.

\begin{thm}\label{3.3}
For a zero-dimensional space $X$,  let $f\in C_c(\beta_0 X)$ and $Z(f)\cap X=\emptyset$. Then $Z(f)$ contains a copy of $\beta \Bbb{N}$, the Stone-$\breve{\mbox{C}}$ech compactification of $\Bbb{N}$, and therefore its cardinality is at least $2^{2^{\aleph_0}}$.
\end{thm}

\begin{proof}
Denote the restriction of $f$ to $X$ by $g$. So $Z(g)=\emptyset$ and $g$ belongs to $C_c(X)$. This function has an inverse in $C_c(X)$, say $\frac{1}{g}$. Obviously  $\frac{1}{g}$ is unbounded on $X$. Hence by Corollary \ref{2.7}, there exists a $C_c$-embedded copy of $\Bbb{N}$ in $X$ on which  the function  $\frac{1}{g}$ tends to infinity.  Since $\Bbb{N}$ is $C_c$-embedded in $X$,  by Lemma \ref{3.1}, it is also two-embedded in $X$ and therefore two-embedded in $\beta_0 X$. Thus every continuous two-valued function on $\Bbb{N}$ is two-embedded in $cl_{\beta_0 X}\Bbb{N}$. By the uniqueness theorem in the existence of Banaschewski compactification,  $cl_{\beta_0 X}\Bbb{N}$ and $\beta_0 \Bbb{N}$ are homeomorphic, and also we have $cl_{\beta_0 X}\Bbb{N}\setminus \Bbb{N}\subseteq Z(f)$. The discrete space $\Bbb{N}$ is strongly zero-dimensional and hence $\beta_0 \Bbb{N}=cl_{\beta_0 X}\Bbb{N}=\beta \Bbb{N}$. The cardinality of $\beta\Bbb{N}$ equals to $2^{2^{\aleph_0}}$ and therefore $|Z(f)|\geq 2^{2^{\aleph_0}}$.
\end{proof}

Since each $G_{\delta}$-point is a zero-set, Lemma \ref{1.3} and Theorem \ref{3.3}  imply the following corollary.

\begin{cor}\label{3.7}
For a zero-dimensional space $X$, no point $p\in \beta_0 X\setminus X$ is a $G_{\delta}$-point of $\beta_0 X$.
\end{cor}

In the next result, by applying Theorem \ref{3.3}, we give the least cardinality of the remainder $\beta_0X\setminus \upsilon_0X$, whenever $X$ is zero-dimensional and non pseudocompact space.

\begin{prop}\label{3.8} 
Let $X$ be a zero-dimensional and non pseudocompact space. The remainder $\beta_0 X\setminus \upsilon_0 X$ has at least  $2^{2^{\aleph_0}}$ points.
\end{prop}
\begin{proof}
Since $X$ is a zero-dimensional and non pseudocompact space, By Proposition \ref{2.8}, there exists a continuous and onto function $f:\upsilon_0 X\rightarrow \Bbb{N}$. Therefore the function $g=\frac{1}{f}$ is continuous whose image equals to the set $\{\frac{1}{n}: n\in \Bbb{N}\}$. Note that the set $\{\frac{1}{n}: n\in \Bbb{N}\}\cup\{0\}$ is zero-dimensional and compact. Hence by \cite[    4.7, proposition (d) ]{PW}, $g$ has a continuous extension $G:\beta_0 X\rightarrow \{\frac{1}{n}: n\in \Bbb{N}\}\cup\{0\}$. For each $n\in \Bbb{N}$, choose $x_n$ such that $g(x_n)=\frac{1}{f(x_n)}=\frac{1}{n}$.  Each $g^{-1}(\frac{1}{n})$  is clopen in $\upsilon_0 X$ and hence the cluster points of the set $\{x_n: n\in \Bbb{N}\}$ is contained in $\beta_0 X\setminus \upsilon_0 X$. If we consider a cluster point  $p$ of the set $\{x_n: n\in \Bbb{N}\}$, then we observe that $G(p)=0$. Therefore $G\in C_c(\beta_0 X)$ and $Z(G)\neq \emptyset$. Since $Z(G)\cap \upsilon_0 X=\emptyset$, Theorem \ref{3.3} implies that $|Z(G)|\geq  2^{2^{\aleph_0}}$. Thus the cardinality of $\beta_0 X\setminus \upsilon_0 X$ is at least $2^{2^{\aleph_0}}$.
\end{proof}

We recall that the local weight of the space $X$ at a point $p$, denoted $\chi(p, X)$ or $\chi (p)$, is the
least cardinal equal to the cardinal number of a (filter) base for the
neighborhoods of $p$.   By Proposition \ref{3.8}, for a zero-dimensional non pseudocompact space $X$, we have the following corollary.

\begin{cor}\label{3.9}
Let $X$ be a zero-dimensional and non pseudocompact space. Then;\\(a) If $p\in \beta_0 X\setminus \upsilon_0 X$, then $\{p\}$ is not a zero-set of $\beta_0 X\setminus \upsilon_0 X$;\\(b) $\beta_0 X\setminus \upsilon_0 X$ has no isolated point;\\(c) The local weight of $\beta_0 X\setminus \upsilon_0 X$ at each of its points is uncountable.
\end{cor}
\begin{proof}
(a) If $p$ is a zero-set of $\beta_0X\setminus \upsilon_0X$, there exists a sequence $\{U_n: n\in \Bbb{N}\}$ of clopen neighborhoods of $p$ in $\beta_0X$ such that $\{p\}=(\cap_{n=1}^{\infty}U_n)\cap (\beta_0X\setminus \upsilon_0X)$. Also there exists a sequence $\{V_n: n\in \Bbb{N}\}$ of clopen neighborhoods of $p$ in $\beta_0X$ such that $(\cap_{n=1}^{\infty}V_n)\cap  \upsilon_0X=\emptyset$. Hence $\{p\}=(\cap_{n=1}^{\infty}U_n)\cap (\cap_{n=1}^{\infty}V_n)$. Part (a) of Lemma \ref{1.3} implies that there exists $f\in C_c(\beta_0X)$ such that $\{p\}=Z(f)$. But this contradicts with Proposition \ref{3.8}.\\
(b) Since each isolated point is a zero-set, clearly part (a) implies part (b).\\
(c) If for some $p$ the locall weight at $p$ is countable, then $p$ is a $G_{\delta}$-point and hence a zero-set of $\beta_0X\setminus \upsilon_0X$, which contradicts with part (a).
\end{proof}

We recall that a topological space is scattered if each of its nonempty subsets has an isolated point with the relative topology. Lemma \ref{3.2} together with part (b) of Corollary \ref{3.9},  imply the following corollary.
\begin{cor}\label{3.10}
For a zero-dimensional space $X$, if $\beta_0 X$ is scattered then $X$ is pseudocompact.
\end{cor}

In the following, we observe some results about connection between Banaschewski compactification and almost $P$-spaces. We recall that a topological space $X$ is an almost $P$-space if each nonempty $G_{\delta}$-subset of $X$ has nonempty interior, see \cite{L} and \cite{V}.

\begin{prop}\label{3.10}
$\beta_0 X$ is an almost $P$-space if and only if $X$ is pseudocompact and an almost $P$-space.
\end{prop}
\begin{proof}
For the necessity, suppose that $\beta_0 X$ is an almost $P$-space and $X$ is not pseudocompact. For $p\in \beta_0 X\setminus \upsilon_0 X$, there exist a sequence of countable clopen subsets of $X$, say $\{U_n: n\in \Bbb{N}\}$, such that $\cap_{n=1}^{\infty}U_n=\emptyset$ and for each $n\in \Bbb{N}$, $p\in cl_{\beta_0 X}U_n$. Note that for each $n\in \Bbb{N}$, $ cl_{\beta_0 X}U_n$ is a clopen subset of $\beta_0 X$. Hence $G=\cap_{n=1}^{\infty} cl_{\beta_0 X}U_n$ is a $G_{\delta}$-subset in $\beta_0 X$. By our hypothesis, $int_{\beta_0 X}G\neq \emptyset$. Therefore $G$ must intersects $X$. But $G\cap X=\cap_{n=1}^{\infty}U_n=\emptyset$, which is a contradiction. Now we show that $X$ is an almost $P$-space. For a nonempty $G_{\delta}$-set $G$, there exists a sequence $\{V_n: n\in \Bbb{N}\}$ of open subsets of $X$ such that $G=\cap_{n=1}^{\infty}V_n$. For $p\in G$, there exists a sequence of clopen sets $\{O_n: n\in \Bbb{N}\}$ such that for each $n\in \Bbb{N}$, $p\in O_n\subseteq V_n$. For each $n\in \Bbb{N}$, $cl_{\beta_0 X}O_n$ is clopen in $\beta_0 X$ and $p\in T=\bigcap_{n=1}^{\infty}\left(cl_{\beta_0 X}O_n\right)$. So there exists a nonempty open set $V$ in $\beta_0 X$ such that $V\subseteq T$. The set $V\cap X$ is nonempty and is contained entirely in $\cap_{n=1}^{\infty}O_n\subseteq G$. Therefore $X$ is an almost $P$-space.\\
 For the sufficiency, let $H$ be a nonempty $G_{\delta}$-subset of $\beta_0 X$. There exist a sequence $\{V_n: n\in \Bbb{N}\}$ of clopen subsets of $\beta_0 X$  and $p\in H$ such that $p\in \cap_{n=1}^{\infty}V_n\subseteq H$. For each $n\in \Bbb{N}$, $V_n\cap X$ is clopen in $X$ and the collection $\Omega=\{V_n\cap X: n\in \Bbb{N}\}$ has finite intersection property. Since $\beta_0 X=\upsilon_0 X$, $\Omega$ is contained in a clopen ultrafilter with countable intersection property. Hence $\bigcap_{n=1}^{\infty}\left(V_n\cap X\right)$ is nonempty and therefore a $G_{\delta}$-subset of $X$. So there exists a neighborhood $O$ in $\beta_0 X$ such that $O\cap X\subseteq \bigcap_{n=1}^{\infty}\left(V_n\cap X\right)$. By taking closure, we observe that $O\subseteq \cap_{n=1}^{\infty}V_n$. Hence $H$ has nonempty interior.
\end{proof}

In the following theorem, we introduce a class of almost $P$-spaces in connection with  Banaschewski compactification of  locally compact and $\Bbb{N}$-compact spaces.

\begin{thm}\label{3.15}
Let $X$ be a locally compact and $\Bbb{N}$-compact space. Then $\beta_0 X\setminus X$ is an almost $P$-space.
\end{thm}
\begin{proof}
Let $G$ be a nonempty $G_{\delta}$-set in $\beta_0 X\setminus X$ and $p\in G$. There exists a sequence of open sets in $\beta_0 X$, say $\{V_n: n\in \Bbb{N}\}$, such that $G=\left(\bigcap_{n=1}^{\infty}V_n\right)\cap \left(\beta_0 X\setminus X\right)$. Also there exists a sequence of clopen sets $\{U_n: n\in \Bbb{N}\}$ of $X$ such that $\bigcap_{n=1}^{\infty}U_n=\emptyset$ and $p\in \bigcap_{n=1}^{\infty}cl_{\beta_0 X}U_n$. Note that the $G_{\delta}$-set $\bigcap_{n=1}^{\infty}cl_{\beta_0 X} U_n$ is a subset of $\beta_0 X\setminus X$. Obviously $H=\left(\bigcap_{n=1}^{\infty}cl_{\beta_0 X} U_n\right)\cap\left(\bigcap_{n=1}^{\infty}V_n\right)$ is a $G_{\delta}$-subset of $\beta_0 X$ and $p\in H\subseteq \beta_0 X\setminus X$. One could find a sequence of clopen subsets of $\beta_0 X$, say $\{O_n: n\in \Bbb{N}\}$, such that $p\in \bigcap_{n=1}^{\infty}O_n\subseteq H$. By part (a) of Lemma \ref{1.3}, there exists $F\in C_c(\beta_0 X)$ such that $Z(F)=\bigcap_{n=1}^{\infty}O_n$. Since $X$ is locally compact, for each $i\in \Bbb{N}$, there exists an open set $W_i\subseteq X$ such that $cl_{X}W_i$ is compact and for each $x\in W_i$, $F(x)\leq \frac{1}{i}$. It is enough to show that the set $\left(\beta_0 X\setminus X\right)\cap cl_{\beta_0X}\left(\bigcup_{i\in \Bbb{N} }W_i\right)$ is a subset of $Z(F)$. Consider $t\in\left(\beta_0 X\setminus X\right)\cap cl_{\beta_0X}\left(\bigcup_{i\in \Bbb{N} }W_i\right)$.  Each neighborhood $P$ of $t$ in $\beta_0 X$ intersects infinitely many $W_i$'s. To see this, assume that for a neighborhood $P$ of $t$, there exists $n\in \Bbb{N}$ such that for all $m>n$; $P\cap W_m=\emptyset$. This implies that $t\in cl_{\beta_0X}\left(\bigcup_{i=1 }^{n}W_i\right)$. Each $W_i$ has compact closure in $X$, and hence $cl_{\beta_0X}\left(\bigcup_{i=1}^{n}W_i\right)$ is a compact subset of $X$ which implies that $t\in X$ and this is a contradiction with the choice of $t$. Since each neighborhood $P$ of $t$ in $\beta_0 X$ intersects infinitely many $W_i$'s, the function $F$ must vanishes in $t$. Hence $$T=(\beta_0X\setminus X)\cap\left(cl_{\beta_0X}\left(\cup_{i\in \Bbb{N}}W_i\right)\right)\subseteq Z(F).$$ It is easy to see that $T$ is nonempty,  open in $\beta_0 X\setminus X$ and  $T\subseteq G$, which shows that $G$ has nonempty interior. This completes the proof. 
\end{proof}

Using Corollary \ref{3.9} together with Theorem \ref{3.15}, we could find a class of  spaces whose appearance is in Boolean algebras, see for example \cite{FZ}. We recall that a compact zero-dimensional space X is called a Parovi$\breve{\mbox{c}}$enko space if it has the following properties:\\
(a) X has no isolated points. \\
(b) Nonempty $G_{\delta}$-sets have nonempty interiors.\\
(c) $X$ is an $F$-space (i.e., for each $f\in C(X)$, the subsets $pos(f)$ and $neg(f)$ are completely separated).

The following proposition gives us a large class of Parovi$\breve{\mbox{c}}$enko spaces in Banaschewski compactification of a non pseudocompact zero-dimensional space.

\begin{prop}\label{3.17}
Let $Z$ be a zero-set of $\beta_0X$ such that $Z\cap X=\emptyset$. Then $Z$ is a Parovi$\breve{\mbox{c}}$enko space.
\end{prop}
\begin{proof}
Evidently $Z$ is compact and zero-dimensional. Since $Z$ is a $G_{\delta}$-subset of $\beta_0X$, $W=\beta_0X\setminus Z$ is $\sigma$-compact and therefore strongly zero-dinensional, see \cite[16.17]{GJ}. Thus $\beta_0W=\beta W$. The space $W$ is Lindel\"{o}f and hence is $\Bbb{N}$-compact. We could observe that $Z=\beta_0W\setminus W=\beta W\setminus W$. So by Theorem \ref{3.15}, $Z$ must be an almost $P$-space. Also since $W$ is locally compact and $\sigma$-compact, by Theorem 14.27 of \cite{GJ}, $Z$ is an $F$-space. Part (b) of  Corollary \ref{3.9}, implies that $Z$ contains no isolated point. Hence $Z$ must be a  Parovi$\breve{\mbox{c}}$enko space.
\end{proof}

Let $X$ be zero-dimensional, locally compact space which is not pseudocompact. We close this section by giving a lower bound for the cellularity of the subspaces $\beta_0X\setminus\upsilon_0X$ and $\beta_0X\setminus X$. We recall that the cellularity of a space $Y$, denoted by $c(Y)$, is the smallest cardinal number $\kappa$ for which each pairwise disjoint family of nonempty open sets of $Y$ has $\kappa$ or fewer members. Also we remind that for a zero-dimensional space $X$ and a cardinal number $\kappa$, a partition of $X$ with length $\kappa$ is a family $\{U_i: i\in \Bbb{I}\}$ of pairwise disjoint clopen subsets of $X$ whose union is $X$ and $|\Bbb{I}|=\kappa$. Note that by Proposition \ref{2.8}, every zero-dimensional space which is not pseudocompact, has a partition with length $\aleph_0$. The following theorem which is due to Tarski, is needed for our purpose, see \cite{HO}. 

\begin{thm}\label{3.11}[Tarski]
Let $E$ be an infinite set. Then there is a collection $\Re$ of subsets of $E$ such that $|\Re|=|E|^{\aleph_0}$, $|R|=\aleph_0$ for each $R\in \Re$ and the intersection of any two distinct members of $\Re$ is finite.
\end{thm}

\begin{thm}\label{3.11}
Let $X$ be a zero-dimensional and locally compact space which is not pseudocompact. If $X$ has a partition of length $\kappa$, then the cellularity of  subspaces $\beta_0X\setminus \upsilon_0X$ and $\beta_0X\setminus X$ are at least $\kappa^{\aleph_0}$.
\end{thm}

\begin{proof}
We just prove that the cellularity of $\beta_0X\setminus \upsilon_0X$ is at least $\kappa^{\aleph_0}$. The second one can be derived similarly. Let $\{U_i: i\in \Bbb{I}\}$ be a partition of $X$ with $|\Bbb{I}|= \kappa$. For each $i\in \Bbb{I}$, choose a nonempty clopen and compact subset $W_i\subseteq U_i$. Note that $\{cl_{\upsilon_0X}U_i: i\in \Bbb{I}\}$ is a partition of length $\kappa$ for $\upsilon_0X$ and for each $i\in \Bbb{I}$, $cl_{\upsilon_0X}W_i=W_i$ is clopen in $\upsilon_0X$. For each subset $J\subseteq \Bbb{I}$, define $$A(J)=(\beta_0X\setminus \upsilon_0X)\cap \left(cl_{\beta_0X}(\bigcup_{i\in J}W_i)\right).$$ We remind that since for each $J\subseteq \Bbb{I}$, $\{W_i: i\in J\}$ is a locally finite family of clopen subsets of $\upsilon_0X$, then $\bigcup_{i\in J}W_i$ is clopen in $\upsilon_0X$ (see \cite[Theorem 1.1.11]{E}) and hence $cl_{\beta_0X}(\bigcup_{i\in J}W_i)$ is clopen in $\beta_0X$. This implies that for each subset $J\subseteq \Bbb{I}$, $A(J)$ is clopen in $\beta_0X\setminus \upsilon_0X$. It is clear that for each two subsets $J_1, J_2$ of $\Bbb{I}$, $A(J_1)\cap A(J_2)=A(J_1\cap J_2)$. For a subset $J\subseteq \Bbb{I}$, we claim that $A(J)=\emptyset$ if and only if $J$ is finite. For, if $J\subseteq \Bbb{I}$ is finite, then $cl_{\beta_0X}(\bigcup_{i\in J}W_i)=\bigcup_{i\in J}W_i$ and hence $A(J)=\emptyset$. For the converse, if $J\subseteq \Bbb{I}$ is infinite, then for each $j\in J$, choose $x_j\in W_j$. Clearly the subset $B=\{x_j: j\in J\}$ is a closed and discrete subset of $\upsilon_0X$. Since $\beta_0X$ is compact, the subset $B$ has a cluster point $p$ in $\beta_0X\setminus \upsilon_0X$. This implies that $A(J)\neq \emptyset$.\\
Now apply Theorem \ref{3.11} to find a collection $\Upsilon$ of $\kappa^{\aleph_0}$ infinite subsets of $\Bbb{I}$ such that any two members of $\Upsilon$ have finite intersection. Define $\mathcal{T}=\{A(J): J\in \Upsilon\}$. Evidently $\mathcal{T}$ contains $\kappa^{\aleph_0}$ pairwise disjoint clopen subsets of $\beta_0X\setminus \upsilon_0X$. This implies that $c(\beta_0X\setminus \upsilon_0X)\geq \kappa^{\aleph_0}$.
\end{proof}

By applying Proposition \ref{2.8} and Theorem \ref{3.11}, the following corollary is immediate.

\begin{cor}\label{3.12}
Let $X$ be a zero-dimensional, locally compact space which is not pseudocompact. Then the cellularity of spaces $\beta_0X\setminus \upsilon_0X$ and $\beta_0X\setminus X$ are at least $2^{\aleph_0}$.
\end{cor}

%------------------------------------------------------------------------------------

%\end{large}
\end{document}